\documentclass[11pt]{article}
\usepackage{amsfonts,amsmath}
\usepackage{hyperref,amsthm}

\newtheorem{theorem}{Theorem}[section]
\newtheorem{definition}[theorem]{Definition}
\newtheorem{proposition}[theorem]{Proposition}
\newtheorem{lemma}[theorem]{Lemma}
\newtheorem{claim}[theorem]{Claim}

\newtheorem{notation}[theorem]{Notation}

\newtheorem{observation}[theorem]{Observation}

\renewenvironment{proof}[1][Proof]{ \noindent \textbf{#1: }}{$\Box$
\bigskip}

\begin{document}

\title{A Tight Quantitative Version of Arrow's Impossibility Theorem}

\author{
Nathan Keller\thanks{The author was partially supported by the
Adams Fellowship Program of the Israeli Academy of Sciences and
Humanities and by
the Koshland Center for Basic Research.}\\
Faculty of Mathematics and Computer Science\\
Weizmann Institute of Science\\
P.O. Box~26, Rehovot 76100, Israel\\
{\tt nathan.keller@weizmann.ac.il}\\
}

\maketitle

\begin{abstract}
The well-known Impossibility Theorem of Arrow asserts that any
Generalized Social Welfare Function (GSWF) with at least three
alternatives, which satisfies Independence of Irrelevant
Alternatives (IIA) and Unanimity and is not a dictatorship, is
necessarily non-transitive. In 2002, Kalai asked whether one can
obtain the following quantitative version of the theorem: For any
$\epsilon>0$, there exists $\delta=\delta(\epsilon)$ such that if
a GSWF on three alternatives satisfies the IIA condition and its
probability of non-transitive outcome is at most $\delta$, then
the GSWF is at most $\epsilon$-far from being a dictatorship or
from breaching the Unanimity condition. In 2009, Mossel proved
such quantitative version, with
$\delta(\epsilon)=\exp(-C/\epsilon^{21})$, and generalized it to
GSWFs with $k$ alternatives, for all $k \geq 3$.

In this paper we show that the quantitative version holds with
$\delta(\epsilon)=C \cdot \epsilon^3$, and that this result is
tight up to logarithmic factors. Furthermore, our result (like
Mossel's) generalizes to GSWFs with $k$ alternatives. Our proof is
based on the works of Kalai and Mossel, but uses also an
additional ingredient: a combination of the Bonami-Beckner
hypercontractive inequality with a reverse hypercontractive
inequality due to Borell, applied to find simultaneously upper
bounds and lower bounds on the ``noise correlation'' between
Boolean functions on the discrete cube.

\end{abstract}

\section{Introduction}
\label{sec:Introduction}

Consider an election procedure in which a society of $n$ members
selects a ranking amongst $k$ alternatives. In the voting process,
each member of the society gives a ranking of the alternatives
(the ranking is a full linear ordering; that is, indifference
between alternatives is not allowed). The set of the rankings
given by the individual members is called a {\it profile}. Given
the profile, the ranking of the society is determined according to
some function, called a {\it generalized social welfare function}
(GSWF).

The GSWF is a function $F:(S_k)^n \rightarrow
\{0,1\}^{{{k}\choose{2}}}$, where $S_k$ is the set of linear
orderings on $k$ elements. In other words, given the profile
consisting of linear orderings supplied by the voters, the
function determines the preference of the society amongst each of
the ${{k}\choose{2}}$ pairs of alternatives. If the output of $F$
can be represented as a full linear ordering of the $k$
alternatives, then $F$ is called a {\it social welfare function}
(SWF).

Throughout this paper we consider GSWFs satisfying the {\it
Independence of Irrelevant Alternatives} (IIA) condition: For any
pair of alternatives $A$ and $B$, the preference of the entire
society between $A$ and $B$ depends only on the preference of each
individual voter between $A$ and $B$. This natural condition on
GSWFs can be traced back to Condorcet~\cite{Condorcet}.

The Condorcet's paradox demonstrates that if the number of
alternatives is at least three and the GSWF is based on the
majority rule amongst every pair of alternatives, then there exist
profiles for which the voting procedure cannot yield a full order
relation. That is, there exist alternatives $A,B,$ and $C$, such
that the majority of the society prefers $A$ over $B$, the
majority prefers $B$ over $C$, and the majority prefers $C$ over
$A$. Such situation is called {\it non-transitive outcome} of the
election.

In his well-known Impossibility theorem~\cite{Arrow}, Arrow showed
that such paradox occurs for any ``reasonable'' GSWF satisfying
the IIA condition:
\begin{theorem}[Arrow]
Consider a generalized social welfare function $F$ with at least
three alternatives. If the following conditions are satisfied:
\begin{itemize}
\item The IIA condition,

\item Unanimity --- if all the members of the society prefer some
alternative $A$ over another alternative $B$, then $A$ is
preferred over $B$ in the outcome of $F$,

\item $F$ is not a dictatorship (that is, the preference of the
society is not determined by a single member),
\end{itemize}
then the probability of a non-transitive outcome is positive
(i.e., there necessarily exists a profile for which the outcome is
non-transitive).
\end{theorem}

Since the existence of profiles leading to a non-transitive
outcome has significant implications on voting procedures, an
extensive research has been conducted in order to evaluate the
probability of non-transitive outcome for various GSWFs. Most of
the results in this area are summarized in~\cite{Gehrlein}. In
addition to its significance in Social Choice theory, this area of
research leads to interesting questions in probabilistic and
extremal combinatorics (see~\cite{Maj-Stablest}).

In 2002, Kalai~\cite{Kalai-Choice} suggested an analytic approach
to this study. He showed that for GSWFs on three alternatives
satisfying the IIA condition, the probability of a non-transitive
outcome with respect to a uniform distribution of the individual
preferences can be computed by a formula related to the
Fourier-Walsh expansion of the GSWF. Using this formula he
presented a new proof of Arrow's impossibility theorem under
additional assumption of neutrality (i.e., invariance of the GSWF
under permutation of the alternatives), and established upper
bounds on the probability of non-transitive outcome for specific
classes of GSWFs.

While providing an analytic proof to Arrow's theorem does not seem
such an important goal (as there are several simple proofs of it,
see~\cite{Geneakoplos}), Kalai aimed at establishing a {\it
quantitative} version of the theorem. Such version would show that
for any $\epsilon>0$, there exists $\delta=\delta(\epsilon)$ such
that if a GSWF on three alternatives satisfies the IIA condition
and its probability of non-transitive outcome is at most $\delta$,
then the GSWF is at most $\epsilon$-far from being a dictatorship
or from breaching the Unanimity condition. Kalai indeed proved
such statement for neutral GSWFs on three alternatives, with
$\delta(\epsilon)=C \cdot \epsilon$ for a universal constant $C$.

Kalai~\cite{Kalai-Private} asked whether his techniques can be
extended to general GSWFs, and suggested to use the Bonami-Beckner
hypercontractive inequality~\cite{Bonamie,Beckner} in order to get
such an extension. However, Keller~\cite{Keller-Choice} showed by
an example that a direct extension cannot hold -- if there exists
$\delta(\epsilon)$ as above, then it cannot depend linearly on
$\epsilon$. Keller asked whether for general GSWFs on three
alternatives, a quantitative version holds with
$\delta(\epsilon)=C \cdot \epsilon^2$.

A few months ago Mossel~\cite{Mossel-Arrow} succeeded to prove a
quantitative version of Arrow's theorem for general GSWFs on three
alternatives. Furthermore, he generalized his result to GSWFs on
more than three alternatives, and to more general probability
distributions on the individual preferences. Unlike Kalai's
techniques, Mossel's proof is quite complex. While Kalai's proof
uses only simple analytic tools but no combinatorial tools,
Mossel's proof extends and exploits a combinatorial proof of
Arrow's theorem given by Barbera~\cite{Barbera}. Furthermore, it
uses ``heavier'' analytic tools, including a reverse
hypercontractive inequality of Borell~\cite{Borell} and a
non-linear invariance principle introduced by Mossel et
al.~\cite{Maj-Stablest}. The only drawback in Mossel's result is
the dependence of $\delta$ on $\epsilon$:
$\delta(\epsilon)=\exp(-C/\epsilon^{21})$ for a universal constant
$C$, which seems far from being optimal. Mossel conjectured that
the ``correct'' dependence of $\delta$ on $\epsilon$ should be
polynomial.\footnote{We note that Mossel also obtained another
variant of his theorem, in which the dependence of $\delta$ on
$\epsilon$ is $\delta(\epsilon)=C \epsilon^3 n^{-3}$, where $n$ is
the number of voters, and $C$ is a ``decent'' constant. As follows
from our results presented below, this variant is essentially
tight for very small values of $\epsilon$ (dependent on $n$).
Moreover, for certain choices of parameters (specifically,
``relatively small'' $n$ and $\epsilon$ very small as a function
of $n$), this result gives a stronger bound than our result, due
to the better value of the constant.}

\bigskip

In this paper we present a tight quantitative version of Arrow's
theorem for general GSWFs. We show that the dependence of $\delta$
on $\epsilon$ is indeed polynomial, and compute the exact
dependence, up to logarithmic factors.

Before we present our results, we should specify the notion of
``the distance of a GSWF on $k$ alternatives satisfying the IIA
condition from a dictatorship or from breaching the Unanimity
condition''. We consider two different definitions of this notion.
In both definitions, the underlying probability measure is the
uniform measure on $(S_k)^n$ (the set of all possible profiles).

The first definition measures the distance of the GSWF under
examination from the family of GSWFs on $k$ alternatives which
satisfy the IIA condition and whose output is {\it always
transitive.} This family was partially characterized by
Wilson~\cite{Wilson}, and fully characterized by
Mossel~\cite{Mossel-Arrow}. It essentially consists of
combinations of dictatorships with constant functions (see
Section~\ref{sec:sub:preliminaries-mossel} for the exact
characterization).
\begin{definition}
Denote by $\mathcal{F}_k(n)$ the family of GSWFs on $k$
alternatives which satisfy the IIA condition and whose output is
always transitive. For a GSWF $F$ on $k$ alternatives that
satisfies the IIA condition, let
\[
D_1(F)=\min_{G \in \mathcal{F}_k(n)} \Pr[F \neq G].
\]
\end{definition}
We note that this is the definition that was used
in~\cite{Mossel-Arrow}. Our main result with respect to this
definition is the following:
\begin{theorem}\label{Thm:Main1}
There exists an absolute constant $C$ such that for any $k$ and
for any GSWF $F$ on $k$ alternatives that satisfies the IIA
condition, if the probability of non-transitive outcome in $F$ is
at most
\[
\delta(\epsilon)=C \cdot \left(\epsilon/k^2 \right)^3,
\]
then $D_1(F) \leq \epsilon$.
\end{theorem}

For the second definition, we note that a GSWF $F$ on $k$
alternatives that satisfies the IIA condition actually consists of
${{k}\choose{2}}$ independent Boolean functions $F_{ij}$ that
represent the choice functions amongst the pairs of alternatives
$(i,j)$ (for $1 \leq i<j \leq k$). The second definition is given
in terms of these functions.
\begin{definition}
Denote by $\mathcal G_2(n)$ the set of constant functions and
dictatorships on two alternatives. For a GSWF $F$ on $k$
alternatives that satisfies the IIA condition, let
\[
D_2(F)=\min_{1 \leq i<j \leq k} \min_{G \in \mathcal G_2(n)}
\Pr[F_{ij} \neq G],
\]
where $\{F_{ij}\}_{1 \leq i<j \leq k}$ are as defined above.
\end{definition}
Our main result with respect to this definition is the following:
\begin{theorem}\label{Thm:Main2}
There exists an absolute constant $C$ such that for any $k$ and
for any GSWF $F$ on $k$ alternatives that satisfies the IIA
condition, if the probability of non-transitive outcome in $F$ is
at most
\begin{equation}\label{Eq1.11}
\delta(\epsilon)=C \cdot \epsilon^{\frac{9(\sqrt{\log_2
(1/\epsilon)}+1/3)^2}{8\log_2(1/\epsilon)}},
\end{equation}
then $D_2(F) \leq \epsilon$.
\end{theorem}
Note that for small values of $\epsilon$, the exponent of
$\epsilon$ in~(\ref{Eq1.11}) tends to $9/8$.

\medskip

We show that the dependence of $\delta$ on $\epsilon$ in
Theorems~\ref{Thm:Main1} and~\ref{Thm:Main2} is tight, up to
logarithmic factors in $\epsilon$. The examples showing the
tightness are GSWFs on three alternatives, in which all the three
choice functions $F_{12},F_{23},$ and $F_{13}$ are monotone
threshold functions. In the example of Theorem~\ref{Thm:Main1},
the expectations of the choice functions are
$0,1-\epsilon,1-\epsilon$ (in particular, one of the functions is
constant!), and in the example of Theorem~\ref{Thm:Main2}, the
expectations are $\epsilon,1/2,1-\epsilon$.

\medskip

As in the works of Kalai and Mossel, the techniques we use are
mainly analytic. Our proof essentially consists of three steps:
\begin{enumerate}
\item We consider a GSWF $F$ on three alternatives, and use a
modification of Kalai's formula to express the probability of
non-transitive outcome as a linear combination of ``noise
correlations'' between the Boolean functions $F_{12},F_{23},$ and
$F_{13}$ (see Section~\ref{sec:sub:preliminaries-noise} for the
definition of noise correlation).

\item We show that if at least one of the functions
$F_{12},F_{23},$ and $F_{13}$ is close enough to a constant
function, then the Bonami-Beckner hypercontractive
inequality~\cite{Bonamie,Beckner} and a reverse hypercontractive
inequality due to Borell~\cite{Borell} can be applied to obtain
simultaneously upper bounds and lower bounds on the noise
correlations. Combination of these bounds yields a lower bound on
the probability of non-transitive outcome in terms of $D_1(F)$ or
$D_2(F)$.

\item To complete the proof, we use the techniques of Mossel to
``cover'' all the remaining cases (i.e., functions with $D_1(F)$
or $D_2(F)$ greater than a fixed constant, etc.)
\end{enumerate}

We note that since in the case where $D_1(F)$ or $D_2(F)$ is
greater than a fixed constant we use Mossel's result as a black
box, the value of the constant we obtain in the dependence of
$\delta(\epsilon)$ on $\epsilon$ is extremely low, and seems to be
very far from optimality. Extension of our techniques to cover all
the cases would make the proof free of non-linear invariance
arguments, and lead to a ``decent'' value of the constant. This is
one of the main open problems left in our paper.

\medskip

This paper is organized as follows: In
Section~\ref{sec:Preliminaries} we present the tools used in the
later sections. In Section~\ref{sec:Lemma} we prove our main
lemma. We deduce Theorems~\ref{Thm:Main1} and~\ref{Thm:Main2} from
the main lemma in Section~\ref{sec:Proof}. In
Section~\ref{sec:Tightness} we discuss the tightness of our
results. We conclude the paper with questions for further research
in Section~\ref{sec:Open-Questions}.

\section{Preliminaries}
\label{sec:Preliminaries}

In this section we present the tools used in the next sections.
First we describe the Fourier-Walsh expansion of functions on the
discrete cube. We continue with the noise operator and the
hypercontractive inequalities of Bonami-Beckner and of Borell.
Finally, we cite the statements from Mossel's proof of the
quantitative Arrow theorem~\cite{Mossel-Arrow} that are used as a
black box in our proof.

\subsection{Fourier-Walsh Expansion of Functions on the Discrete
Cube}
\label{sec:sub:Fourier}

Throughout the paper we consider the discrete cube
$\Omega=\{0,1\}^n$, endowed with the uniform measure $\mu$.
Elements of $\Omega$ are represented either by binary vectors of
length $n$, or by subsets of $\{1,2,\ldots,n\}$. Denote the set of
all real-valued functions on $\Omega$ by $X$. The inner product of
functions $f,g \in X$ is defined as usual as
\[
\langle f,g \rangle = \mathbb{E}_{\mu} [fg] = \frac{1}{2^n}
\sum_{x \in \{0,1\}^n} f(x)g(x).
\]
The Rademacher functions $\{r_i\}_{i=1}^n$, defined as
$r_i(x_1,\ldots,x_n)=2x_i-1$, constitute an orthonormal system in
$X$. Moreover, this system can be completed to an orthonormal
basis in $X$ by defining
\[
r_S = \prod_{i \in S} r_i,
\]
for all $S \subset \{1,\ldots,n\}$. Every function $f \in X$ can
be represented by its Fourier expansion with respect to the system
$\{r_S\}_{S \subset \{1,\ldots,n\}}$:
\[
f = \sum_{S \subset \{1,\ldots,n\}} \langle f,r_S \rangle r_S.
\]
This representation is called the Fourier-Walsh expansion of $f$.
The coefficients in this expansion are denoted by
\[
\hat f(S) = \langle f,r_S \rangle.
\]

\medskip

\noindent The Fourier-Walsh expansion allows to adapt tools from
classical harmonic analysis to the discrete setting, and to use
them in the study of Boolean functions. Since the introduction of
such analytic methods in the landmark paper of Kahn, Kalai, and
Linial~\cite{KKL} in 1988, they were intensively studied, and led
to applications in numerous fields, including combinatorics,
theoretical computer science, social choice theory, mathematical
physics, etc. (see, e.g., the survey~\cite{Kalai-Safra}).

\medskip

\noindent The most basic analytic tool we use is the Parseval
identity, asserting that for all $f,g \in X$,
\[
\langle f,g \rangle = \sum_{S \subset \{1,\ldots,n\}} \hat f(S)
\hat g(S),
\]
and in particular, $\sum_{S \subset \{1,\ldots,n\}} \hat f(S)^2 =
||f||_2^2$, for any $f \in X$.

\medskip

\noindent The next simple tool we use is the close relation
between the Fourier-Walsh expansions of a function and of the
respective dual function.
\begin{definition}
Let $f:\{0,1\}^n \rightarrow \{0,1\}$. The dual function of $f$,
which we denote by $\bar{f}:\{0,1\}^n \rightarrow \{0,1\}$, is
defined by
\[
\bar{f}(x_1,x_2,\ldots,x_n)=1-f(1-x_1,1-x_2,\ldots,1-x_n).
\]
\end{definition}
\begin{claim}\label{Claim:Dual1}
Consider the Fourier-Walsh expansions of a Boolean function $f$
and its dual function $\bar{f}$. For any $S \subset
\{1,\ldots,n\}$ with $|S| \geq 1$,
\begin{equation}\label{Eq:Dual1}
\widehat{\bar{f}}(S) = (-1)^{|S|-1} \hat f(S).
\end{equation}
\end{claim}

\noindent The simple proof of the claim is omitted. We use also a
variant of the dual function: $f'(x)=1-\bar{f}(x)$, defined as
\[
f'(x_1,x_2,\ldots,x_n)=f(1-x_1,1-x_2,\ldots,1-x_n).
\]
Similarly to Claim~\ref{Claim:Dual1}, it is easy to see that for
any $|S| \geq 1$,
\begin{equation}\label{Eq:Dual2}
\hat{f}'(S) = (-1)^{|S|} \hat{f}(S).
\end{equation}

\noindent In Kalai's proof of the quantitative Arrow theorem for
neutral GSWFs~\cite{Kalai-Choice}, only the most basic analytic
tools (like the Parseval identity) were used. Following the proof
of Mossel~\cite{Mossel-Arrow}, we use also more advanced analytic
tools, related to the noise operator presented below.


\subsection{The Noise Operator and Hypercontractive Inequalities}
\label{sec:sub:preliminaries-noise}

The {\it noise operator}, defined in~\cite{Beckner,Bonamie}, is a
convolution operator that represents the application of the
function on a slightly perturbed input.
\begin{definition}
For $x \in \{0,1\}^n$, the $\epsilon$-noise perturbation of $x$,
denoted by $N_\epsilon(x)$, is a distribution obtained from $x$ by
independently keeping each coordinate of $x$ unchanged with
probability $1-\epsilon$, and replacing it by a random value with
probability $\epsilon$.
\end{definition}
\begin{definition}
Let $f:\{0,1\}^n \rightarrow \{0,1\}$. For $0 \leq \epsilon \leq
1$, the noise operator $T_{\epsilon}$ applied to $f$ is defined by
\[
T_{\epsilon}f(x)=\mathbb{E}_{y \sim N_{1-\epsilon}(x)} [f(y)].
\]
\end{definition}

\noindent It is easy to see that the noise operator has a
convenient representation in terms of the Fourier-Walsh expansion:
\begin{claim}\label{Claim:Noise1}
Consider a function $f$ on the discrete cube with a Fourier-Walsh
expansion $f=\sum_S \hat f(S) r_S$. The Fourier-Walsh expansion of
$T_{\epsilon}f$ is given by:
\begin{equation}
T_{\epsilon} f = \sum_S \epsilon^{|S|} \hat f(S) r_S.
\label{Eq:Beckner0}
\end{equation}
\end{claim}

\noindent Since $T_{\epsilon} f$ represents the application of $f$
on a noisy variant of the input, it makes sense to define the {\it
$\epsilon$-noise correlation} of two functions $f$ and $g$ as
$\langle T_{\epsilon}f,g \rangle$. Using the Parseval identity, we
get an equivalent definition in terms of the Fourier-Walsh
expansion (note that the definition is symmetric between $f$ and
$g$):
\begin{definition}
Given two functions $f,g:\{0,1\}^n \rightarrow \{0,1\}$, the
$\epsilon$-noise correlation of $f,g$ is
\begin{equation}\label{Eq:Noisy1}
\langle T_{\epsilon} f,g \rangle = \sum_{S \subset
\{1,2,\ldots,n\}} \epsilon^{|S|} \hat f(S) \hat g(S).
\end{equation}
\end{definition}

\noindent In the proof of Lemma~\ref{Lemma:Main}, we express the
probability of non-transitive outcome in a GSWF $F$ on three
alternatives in terms of the noise correlations between the
Boolean choice functions $F_{12},F_{23},$ and $F_{13}$. Then we
obtain upper and lower bounds on the noise correlations using the
hypercontractive inequalities presented below.

\medskip

\noindent The first hypercontractive inequality we use is the
Bonami-Beckner inequality, discovered independently by
Bonami~\cite{Bonamie} in 1970 and by Beckner~\cite{Beckner} in
1975.
\begin{theorem}[Bonami,Beckner]
Let $f:\{0,1\}^n \rightarrow \mathbb{R}$, and let $q_1 \geq q_2
\geq 1$. Then
\[
||T_{\epsilon}f||_{q_1} \leq ||f||_{q_2}, \qquad \mbox{ for all }
0 \leq \epsilon \leq \left( \frac{q_2-1}{q_1-1} \right)^{1/2}.
\]
In particular,
\[
||T_{\epsilon}f||_2 \leq ||f||_{1+\epsilon^2}, \mbox{ for all  } 0
\leq \epsilon \leq 1.
\]
\end{theorem}

\noindent This inequality was first applied in a combinatorial
context in~\cite{KKL}, and since then it was used in numerous
papers in the field. We combine the Bonami-Beckner inequality with
the Cauchy-Schwarz inequality to obtain an {\it upper bound} on
the $\epsilon$-noise correlation of Boolean functions. The upper
bound is presented here for $\epsilon=1/3$ since this is the case
we use in the proof of Lemma~\ref{Lemma:Main}, but it can be
immediately generalized to any $0 \leq \epsilon \leq 1$.
\begin{proposition}\label{Prop:Upper-Bound1}
Let $f,g:\{0,1\}^n \rightarrow \{0,1\}$, and denote $\mathbb{E}
[f]=p_1$, and $\mathbb{E}[g] = p_2$. Then:
\begin{equation}
\sum_S (\frac{1}{3})^{|S|} \hat f(S) \hat g(S) \leq \min
\left(p_1^{0.9} p_2^{0.5}, p_1^{0.75} p_2^{0.75} \right).
\end{equation}
\end{proposition}
\begin{proof}
By Claim~\ref{Claim:Noise1}, the Parseval identity, the
Cauchy-Schwarz inequality and the Bonami-Beckner hypercontractive
inequality, we get:
\[
\sum_S (\frac{1}{3})^{|S|} \hat f(S) \hat g(S) = \langle T_{1/3}
f, g \rangle \leq ||T_{1/3}f||_2 ||g||_2 \leq ||f||_{1+1/9}
||g||_2 = p_1^{0.9} p_2^{0.5}.
\]
Similarly,
\begin{align*}
\sum_S (\frac{1}{3})^{|S|} \hat f(S) \hat g(S) &= \langle
T_{(1/\sqrt{3})} f, T_{(1/\sqrt{3})} g \rangle \leq
||T_{(1/\sqrt{3})}f||_2 ||T_{(1/\sqrt{3})}g||_2 \\
&\leq ||f||_{1+1/3} ||g||_{1+1/3} = p_1^{0.75} p_2^{0.75}.
\end{align*}
This completes the proof.
\end{proof}

\medskip

\noindent The second hypercontractive inequality we use is a
reverse hypercontractive inequality, due to Borell~\cite{Borell}.
This inequality asserts that under some conditions, a variant of
the Bonami-Beckner inequality holds in the inverse direction.
\begin{theorem}[Borell]
Let $f:\{0,1\}^n \rightarrow \mathbb{R}_+$, and let $q_1 \leq q_2
\leq 1$. Then
\[
||T_{\epsilon}f||_{q_1} \geq ||f||_{q_2}, \qquad \mbox{ for all }
0 \leq \epsilon \leq \left( \frac{q_2-1}{q_1-1} \right)^{1/2}.
\]
\end{theorem}
\noindent Although Borell's result dates back to 1982, it wasn't
used in the research of Boolean functions until recent years. In
the last few years, Borell's inequality was used in several
papers~\cite{Feige,Mossel-Inverse,Mossel-Arrow}; it seems to be a
useful tool that has yet to be fully developed.

\medskip

\noindent We use Borell's inequality to obtain a {\it lower bound}
on the $\epsilon$-noise correlation of Boolean functions through
the following corollary, presented in~(\cite{Mossel-Inverse},
Corollary~3.5):
\begin{theorem}\label{Thm:Inverse-Beckner1}
Let $f,g:\{0,1\}^n \rightarrow \{0,1\}$. If $\mathbb{E}[f] = p_1$,
and $\mathbb{E}[g] = p_2=p_1^{\alpha}$ for some $\alpha \geq 0$,
then for all $0 < \epsilon <1$,
\begin{equation}
\sum_{S} \epsilon^{|S|} \hat f(S) \hat g(S) \geq p_1 \cdot
p_1^{(\sqrt{\alpha}+\epsilon)^2/(1-\epsilon^2)} = p_1 \cdot
p_2^{\frac{(\sqrt{\alpha}+\epsilon)^2}{(1-\epsilon^2) \alpha}}.
\label{Eq:Inv-Beckner1}
\end{equation}
\end{theorem}

\noindent In the proof of Lemma~\ref{Lemma:Main}, we apply
Theorem~\ref{Thm:Inverse-Beckner1} with $\alpha \geq 1$, and write
the lower bound in the form $p_1 \cdot p_2^{\beta}$. We use a
simple observation regarding properties of
$\beta=\frac{(\sqrt{\alpha}+\epsilon)^2}{(1-\epsilon^2) \alpha}$:
\begin{observation}

\begin{itemize}

\item As a function of $\alpha=\log_{p_1} (p_2)$, $\beta$ is
monotone decreasing.

\item For all $\alpha \geq 1$, we have $\beta \leq \frac{
1+\epsilon}{1-\epsilon}$. When $\alpha \rightarrow \infty$, we
have $\beta \rightarrow \frac{1}{1-\epsilon^2}$.
\end{itemize}
\label{Obs:Inv-Beckner}
\end{observation}

\noindent Finally, we use the following notation:
\begin{notation}
We denote by $RHC(p_1,p_2)$ the lower bound obtained in
Theorem~\ref{Thm:Inverse-Beckner1} for $\mathbb{E}[f]=p_1$,
$\mathbb{E}[g]=p_2$, and $\epsilon=1/3$. In particular,
$RHC(p,p)=p^3$, and
\begin{equation}\label{Eq:RHC'}
RHC(1/2,p)=\frac{1}{2} \cdot
p^{\frac{9(\sqrt{\alpha}+1/3)^2}{8\alpha}},
\end{equation}
where $\alpha = \log_2 (1/p)$. For a small value of $p$, the
exponent tends to $9/8$. \label{Not:RHC}
\end{notation}

\subsection{Mossel's Quantitative Arrow Theorem}
\label{sec:sub:preliminaries-mossel}

In the proofs of Theorems~\ref{Thm:Main1} and~\ref{Thm:Main2} we
use (as a ``black box'') three major components of Mossel's proof
of his quantitative version of Arrow's
theorem~\cite{Mossel-Arrow}. The first is a quantitative Arrow
theorem for GSWFs on three alternatives:
\begin{theorem}[\cite{Mossel-Arrow}, Theorem~8.1] \label{Thm:Mossel1}
There exists an absolute constant
$C$ such that for any GSWF $F$ on three alternatives that
satisfies the IIA condition, if the probability of non-transitive
outcome in $F$ is at most
\[
\delta(\epsilon)= \exp(-C/\epsilon^{21}),
\]
then $D_1(F) \leq \epsilon$.
\end{theorem}

\noindent We use this theorem only in the case where $\epsilon$ is
bigger than some fixed constant. Thus, the ``bad'' dependence of
$\delta$ on $\epsilon$ affects our final result only by a constant
factor.

\medskip

\noindent The second component is a generic reduction lemma that
allows to leverage results from GSWFs on three alternatives to
GSWFs on $k$ alternatives, for all $k \geq 3$. The reduction can
be formulated as follows:
\begin{theorem}[\cite{Mossel-Arrow}, Theorem~9.1 and Remark~9.2]
\label{Thm:Leveraging} Suppose that there exists
$\delta_0(\epsilon)$ such that for any GSWF $F$ on three
alternatives that satisfies the IIA condition, if the probability
of non-transitive outcome in $F$ is at most $\delta_0(\epsilon)$,
then $D_1(F) \leq \epsilon$. Then we have the following
quantitative Arrow theorem for GSWFs on $k$ alternatives:

\noindent For any GSWF $F$ on $k$ alternatives that satisfies the
IIA condition, if the probability of non-transitive outcome in $F$
is at most
\[
\delta(\epsilon)=\delta_0 \left(\epsilon/k^2 \right),
\]
then $D_1(F) \leq \epsilon$.
\end{theorem}

\medskip

\noindent The third component is a complete characterization of
the set $\mathcal{F}_k(n)$ of GSWFs on $k$ alternatives that
satisfy the IIA condition and whose output is always transitive.
Though we use in our proof only the characterization of
$\mathcal{F}_3(n)$, the result is presented here for a general $k$
for the sake of completeness.
\begin{theorem}[\cite{Mossel-Arrow}, Theorem~1.2]
\label{Thm:Mossel3}
The class $\mathcal{F}_k(n)$ consists exactly
of all GSWFs $F$ satisfying the following: There exists a
partition of the set of alternatives into disjoint sets
$A_1,A_2,\ldots,A_r$ such that:
\begin{itemize}
\item For any profile, $F$ ranks all the alternatives in $A_i$
above all the alternatives in $A_j$, for all $i<j$.

\item For all $s$ such that $|A_s| \geq 3$, the restriction of $F$
to the alternatives in $A_s$ is a dictatorship (i.e., is equal
either to the preference order of some voter $j$ or to the reverse
of such order).

\item For all $s$ such that $|A_s|=2$, the restriction of $F$ to
the alternatives in $A_s$ is an arbitrary non-constant function of
the individual preferences between the two alternatives in $A_s$.
\end{itemize}
\end{theorem}

\section{Proof of the Main Lemma}
\label{sec:Lemma}

In this section we prove our main lemma, asserting that if $F$ is
a GSWF on three alternatives that satisfies the IIA condition, and
at least one of the Boolean choice functions $F_{12},F_{23},$ and
$F_{31}$ is ``close enough'' to a constant function, then the
probability of non-transitive outcome can be bounded from below in
terms of $D_1(F)$ and $D_2(F)$. Throughout this section we use the
following notations:
\begin{notation}
The Boolean choice functions $F_{12},F_{23},$ and $F_{31}$ are
denoted by $f,g,$ and $h$, respectively. This means that the
preferences of the voters between alternatives $1$ and $2$ are
denoted by a vector $(x_1,x_2,\ldots,x_n) \in \{0,1\}^n$, where
$x_k=1$ if the $k$'th voter prefers alternative $1$ over
alternative $2$, and $x_k=0$ otherwise. Then,
$f(x_1,\ldots,x_n)=1$ if in the output of $F$, alternative $1$ is
preferred over alternative $2$, and $f(x_1,\ldots,x_n)=0$
otherwise. The functions $g$ and $h$ are defined similarly with
respect to the pairs of alternatives $(2,3)$ and $(3,1)$. The
expectations of the choice functions are denoted by
\[
\mathbb{E}[f]=p_1, \qquad \mathbb{E}[g]=p_2, \qquad
\mathbb{E}[h]=p_3.
\]
For each $i$, we denote $\overline{p}_i=\min(p_i,1-p_i)$, and let
\[
D'_2(F)=\min_{1 \leq i \leq 3} \overline{p}_i.
\]
Note that $D'_2(F)$ measures the distance of the Boolean choice
functions of $F$ from the family of constant functions. Finally,
the probability of non-transitive outcome is denoted by $P(F)$.
\end{notation}

\noindent Now we can formulate our main lemma:
\begin{lemma}\label{Lemma:Main}
Let $F$ be a GSWF on three alternatives satisfying the IIA
condition. If $D'_2(F) \leq 2^{-500000}$, then
\[
P(F) \geq \frac{1}{10} \cdot \max
\Big(RHC(D_1(F)/2,D_1(F)/2),RHC(D'_2(F),1/2) \Big),
\]
where $D_1(F)$ is as defined in the introduction and
$RHC(\cdot,\cdot)$ is as defined in Notation~\ref{Not:RHC}.
\end{lemma}

\begin{proof}
Our starting point is Kalai's formula~\cite{Kalai-Choice} for the probability of
non-transitive outcome in a GSWF on three alternatives satisfying the IIA
condition:
\begin{align}\label{Eq3.1}
\begin{split}
P(F) =& p_1 p_2 p_3 + (1-p_1)(1-p_2)(1-p_3) + \\
&+ \sum_{S \neq
\emptyset} (-\frac{1}{3})^{|S|} \hat{f}(S) \hat{g} (S) + \sum_{S \neq
\emptyset} (-\frac{1}{3})^{|S|} \hat{g}(S) \hat{h} (S) + \sum_{S \neq
\emptyset} (-\frac{1}{3})^{|S|} \hat{h}(S) \hat{f} (S).
\end{split}
\end{align}

\noindent The proof is divided into several cases, and in each case we use a
different modification of Formula~(\ref{Eq3.1}). In the
following, we assume w.l.o.g. that $\overline{p}_1 \leq
\overline{p}_2 \leq \overline{p}_3$. Moreover, we assume w.l.o.g. that
$p_1 \leq 1/2$, since otherwise we can
replace $f,g,h$ by the dual functions without changing the value of the
right hand side of~(\ref{Eq3.1}).

\subsection{Case 1: $p_1,p_2 \leq 1/2$}
\label{sec:sub:Same-Direction}

First, we note that if $p_3 \leq 1/2$, then the assertion follows
easily from Kalai's Formula~(\ref{Eq3.1}). Indeed, since by
assumption we have $p_1<2^{-500000}$, it follows that
$(1-p_1)(1-p_2)(1-p_3) \geq \frac14 (1-2^{-500000})$. On the other
hand, by the Parseval identity and the Cauchy-Schwarz inequality,
we have
\[
|\sum_{S \neq \emptyset} (-\frac{1}{3})^{|S|} \hat{f}(S) \hat{g}
(S)| \leq \frac13 \Big( \langle f,g \rangle - p_1 p_2 \Big) \leq
\frac13 \cdot \frac12 \cdot 2^{-250000},
\]
and similarly,
\[
|\sum_{S \neq \emptyset} (-\frac{1}{3})^{|S|} \hat{g}(S) \hat{h}
(S)| \leq \frac13 \cdot \frac 12 \cdot \frac12, \qquad \mbox{and}
\qquad |\sum_{S \neq \emptyset} (-\frac{1}{3})^{|S|} \hat{h}(S)
\hat{f} (S)| \leq \frac13 \cdot \frac12 \cdot 2^{-250000}.
\]
Thus, by Formula~(\ref{Eq3.1}),
\[
P(F) \geq \frac14 (1-2^{-500000}) - \frac{1}{12} - 2 \cdot \frac13 \cdot
\frac12 \cdot 2^{-250000} > 1/10.
\]
Therefore, we can assume that $p_3 \geq 1/2$. Note that in this
case, we have
\begin{equation}\label{Eq3.2}
D_1(F) \leq 2(1-p_3).
\end{equation}
Indeed, define a GSWF $G$ on three alternatives by the choice
functions:
\[
f'=G_{12}=0 \mbox{ (constant) }, \qquad g'=G_{23}=g, \qquad
h'=G_{31}=1 \mbox{ (constant) }.
\]
It is clear that $G \in \mathcal{F}_3(n)$, since $G$ always ranks
alternative $1$ at the bottom and thus its output is always
transitive, and
\[
\Pr[F \neq G] \leq \Pr[f \neq f']+ \Pr[g \neq g']+\Pr[h \neq h']
\leq p_1 + (1-p_3) \leq 2(1-p_3).
\]
Therefore, $D_1(F) \leq D(F,G) \leq 2(1-p_3)$. Also, by the
definition,
\begin{equation}\label{Eq3.3}
D'_2(F)= p_1.
\end{equation}

\medskip

\noindent We modify Formula~(\ref{Eq3.1}) using the following identities:
\begin{equation}
\sum_{S \neq \emptyset} (-\frac{1}{3})^{|S|} \hat{f}(S) \hat{g} (S) = \sum_{S
\neq \emptyset} (\frac{1}{3})^{|S|} \hat{f}'(S) \hat{g} (S) =
\langle T_{1/3} f' , g \rangle - p_1 p_2.
\label{Eq:Modification1.1}
\end{equation}
\begin{align}\label{Eq:Modification1.2}
\begin{split}
\sum_{S \neq \emptyset} (-\frac{1}{3})^{|S|} \hat{g}(S) \hat{h} (S) &=
\sum_{S \neq \emptyset}(-\frac{1}{3})^{|S|} \widehat{\bar{g}}(S)
\widehat{\bar{h}} (S) = \sum_{S \neq \emptyset}(\frac{1}{3})^{|S|}
\widehat{\bar{g}}(S) \widehat{1-h} (S) \\
&= \langle T_{1/3} \bar{g} , 1-h \rangle - (1-p_2)(1-p_3).
\end{split}
\end{align}
\begin{align}\label{Eq:Modification1.3}
\begin{split}
\sum_{S \neq \emptyset} (-\frac{1}{3})^{|S|} \hat{h}(S) \hat{f} (S) &= \langle T_{1/3} f' ,
h \rangle - p_1 p_3 = \langle T_{1/3} f' , 1 \rangle - \langle
T_{1/3} f' , 1-h \rangle - p_1 p_3 \\
&= p_1 - p_1 p_3 - \langle T_{1/3} f' , 1-h \rangle.
\end{split}
\end{align}
All three identities follow immediately from basic properties
of the dual function (e.g., Equations~(\ref{Eq:Dual1}) and~(\ref{Eq:Dual2})) and the
Parseval identity. Substituting
Equations~(\ref{Eq:Modification1.1}),~(\ref{Eq:Modification1.2}),
and~(\ref{Eq:Modification1.3}) into Formula~(\ref{Eq3.1}), we
get:
\begin{align}\label{Eq:Modification1.4}
\begin{split}
P(F) =& p_1 p_2 p_3 + (1-p_1)(1-p_2)(1-p_3) + \\
&+ \langle T_{1/3} f' , g \rangle - p_1 p_2 + \langle T_{1/3} \bar{g} , 1-h \rangle
- (1-p_2)(1-p_3) + p_1 - p_1 p_3 - \langle T_{1/3} f' , 1-h \rangle  = \\
=& \langle T_{1/3} f' , g \rangle + \langle T_{1/3} \bar{g} , 1-h \rangle -
\langle T_{1/3} f' , 1-h \rangle.
\end{split}
\end{align}
Equation~(\ref{Eq:Modification1.4}) expresses $P(F)$ as a linear
combination of ``noise correlations'' between the functions
$f,g,h$, which are obviously nonnegative. Thus, if we obtain a
lower bound on the noise correlations that appear in
Equation~(\ref{Eq:Modification1.4}) with a `+' sign, and an upper
bound on the correlation that appears with a `-' sign, we will get
a lower bound on $P(f)$. We shall obtain these bounds using the
Bonami-Beckner hypercontractive inequality and Borell's reverse
hypercontractive inequality. We subdivide our case into two
sub-cases.

\subsubsection{Case 1a: $1-p_3<1/32$.}

We bound $\langle T_{1/3} f' , 1-h \rangle$ from above
using the Bonami-Beckner hypercontractive inequality. By
Proposition~\ref{Prop:Upper-Bound1}, we get:
\begin{equation}
\langle T_{1/3} f' , 1-h \rangle \leq p_1^{0.75} (1-p_3)^{0.75}
\leq (1-p_3)^{1.5}. \label{Eq:Modification1.5}
\end{equation}
We bound $\langle T_{1/3} \bar{g} , 1-h \rangle$ from below using
Borell's reverse hypercontractive inequality. By
Theorem~\ref{Thm:Inverse-Beckner1}, we have
\[
\langle T_{1/3} \bar{g} , 1-h \rangle \geq RHC(1-p_2,1-p_3).
\]
In order to estimate $RHC(1-p_2,1-p_3)$, we write it in the form
$(1-p_2) \cdot (1-p_3)^{\beta(\alpha)}$, where $\alpha =
\log_{1-p_2} (1-p_3)$. By Observation~\ref{Obs:Inv-Beckner},
$\beta(\alpha)$ is a monotone decreasing function of $\alpha$.
Since by assumption, $1-p_2>31/32$ and $1-p_3<1/32$, we have
$\alpha \geq \log_{31/32} (1/32) = 109.16$. Substituting the value
$\alpha=109.16$ into the definition of $\beta(\alpha)$ and using
the monotonicity of $\beta(\alpha)$, we get $\beta \leq 1.198$,
and thus,
\begin{equation}
\langle T_{1/3} \bar{g} , 1-h \rangle \geq (1-p_2) (1-p_3)^{1.198}
\geq \frac{31}{32} (1-p_3)^{1.198}. \label{Eq:Modification1.6}
\end{equation}
Combining Inequalities~(\ref{Eq:Modification1.5})
and~(\ref{Eq:Modification1.6}), we get
\begin{align}\label{Eq:Modification1.7}
\begin{split}
\langle T_{1/3} f' , 1-h \rangle \leq (1-p_3)^{1.5} &=
\Big(\frac{32}{31} (1-p_3)^{0.302} \Big) \Big(\frac{31}{32}
(1-p_3)^{1.198} \Big) \\
&\leq \Big(\frac{32}{31} (\frac{1}{32})^{0.302} \Big) \langle
T_{1/3} \bar{g} , 1-h \rangle \leq 0.37 \langle T_{1/3} \bar{g} ,
1-h \rangle.
\end{split}
\end{align}
Finally, substituting into Equation~(\ref{Eq:Modification1.4}) we
get:
\begin{align}
\begin{split}
P(F) &= \langle T_{1/3} f' , g \rangle + \langle T_{1/3}
\bar{g} , 1-h \rangle - \langle T_{1/3} f' , 1-h \rangle \\
&\geq \langle T_{1/3} \bar{g} , 1-h \rangle - \langle T_{1/3} f' ,
1-h \rangle \geq 0.63 \langle T_{1/3} \bar{g} , 1-h \rangle \\
&\geq 0.63 RHC(1-p_2,1-p_3) \geq 0.63 \max
\left(RHC(1-p_3,1-p_3), RHC(1/2,p_1) \right),
\end{split}
\end{align}
where the last inequality holds since $1-p_2 \geq 1/2 \geq 1-p_3
\geq p_1$ and since $RHC(\cdot,\cdot)$ is clearly non-decreasing
in its arguments. The assertion of the lemma follows now from
Inequalities~(\ref{Eq3.2}) and~(\ref{Eq3.3}).

\subsubsection{Case 1b: $1-p_3 \geq 1/32$.}

As in the previous case, we bound $\langle T_{1/3} f' , 1-h
\rangle$ from above using the Bonami-Beckner hypercontractive
inequality. By Proposition~\ref{Prop:Upper-Bound1}, we get:
\begin{equation}
\langle T_{1/3} f' , 1-h \rangle \leq p_1^{0.9} (1-p_3)^{0.5} \leq
(1-p_3)^{4.1} \label{Eq:Modification1.8},
\end{equation}
where the last inequality follows since by assumption $p_1 \leq
2^{-500000}$, and in particular, $p_1 \leq (1-p_3)^4$. In order to
bound $\langle T_{1/3} \bar{g} , 1-h \rangle$ from below we use
the reverse hypercontractive inequality. By
Theorem~\ref{Thm:Inverse-Beckner1}, we have
\[
\langle T_{1/3} \bar{g} , 1-h \rangle \geq RHC(1-p_2,1-p_3).
\]
As in the previous case, we write $RHC(1-p_2,1-p_3)$ in the form
$(1-p_2) \cdot (1-p_3)^{\beta(\alpha)}$. Since by
Observation~\ref{Obs:Inv-Beckner}, for any $\alpha \geq 1$, we have
$\beta(\alpha) \leq \frac{1+\epsilon}{1-\epsilon} =2$, we get
\begin{equation}
\langle T_{1/3} \bar{g} , 1-h \rangle \geq (1-p_2)(1-p_3)^2 \geq
0.5(1-p_3)^2. \label{Eq:Modification1.9}
\end{equation}
Combination of Equations~(\ref{Eq:Modification1.8}) and
(\ref{Eq:Modification1.9}) yields
\begin{align*}
\langle T_{1/3} f' , 1-h \rangle \leq (1-p_3)^{4.1} &=
\Big(2(1-p_3)^{2.1} \Big) \Big(0.5(1-p_3)^2 \Big)
\\ &\leq 2(1-p_3)^{2.1} \langle T_{1/3} \bar{g} , 1-h \rangle \leq 0.5
\langle T_{1/3} \bar{g} , 1-h \rangle,
\end{align*}
where the last inequality follows since $1-p_3 \leq 1/2$. Finally,
\begin{align}
\begin{split}
P(F) &= \langle T_{1/3} f' , g \rangle + \langle T_{1/3}
\bar{g} , 1-h \rangle - \langle T_{1/3} f' , 1-h \rangle \\
&\geq \langle T_{1/3} \bar{g} , 1-h \rangle - \langle T_{1/3} f' , 1-h
\rangle \geq 0.5 \langle T_{1/3} \bar{g} , 1-h \rangle \geq 0.5
RHC(1-p_2,1-p_3) \\
&\geq 0.5 \max \left( RHC(1-p_3,1-p_3),RHC(1/2,p_1) \right),
\end{split}
\end{align}
as asserted. This completes the proof of Case~1.

\subsection{Case 2: $p_1 \leq 1/2$ and $p_2 \geq 1/2$}
\label{sec:sub:Opposite-Directions}

In this case, we have
\begin{equation}\label{Eq3.4}
D_1(F) \leq 2(1-p_2),
\end{equation}
since defining a GSWF $G'$ on three alternatives by the choice
functions:
\[
f''=G'_{12}=0 \mbox{ (constant) }, \qquad g''=G'_{23}=1 \mbox{ (constant) },
\qquad h''=G'_{31}=h,
\]
we get $G' \in \mathcal{F}_3(n)$, and $D(F,G') \leq 2(1-p_2)$.
Also, it is clear that like in Case~1,
\begin{equation}\label{Eq3.5}
D_2(F) \leq D(f,const) = p_1.
\end{equation}

\medskip

\noindent This time we use a slightly different modification of Kalai's
formula. Specifically, we interchange the roles of $g$ and $h$ in
Equations~(\ref{Eq:Modification1.1}),~(\ref{Eq:Modification1.2}),
and~(\ref{Eq:Modification1.3}), and get the following modification
of Equation~(\ref{Eq:Modification1.4}):
\begin{equation}
P(F)= \langle T_{1/3} f' , h \rangle + \langle T_{1/3} (1-g),
\bar{h} \rangle - \langle T_{1/3} f' , 1-g \rangle.
\label{Eq:Modification2.1}
\end{equation}
We subdivide this case into several sub-cases.

\subsubsection{Case 2a: $(1-p_2) \leq p_1^{0.45412}$. }

By Proposition~\ref{Prop:Upper-Bound1}, we get:
\begin{equation}
\langle T_{1/3} f' , 1-g \rangle \leq p_1^{0.9} (1-p_2)^{0.5} \leq
p_1^{0.9} p_1^{0.22706} = p_1^{1.12706}.
\label{Eq:Modification2.2}
\end{equation}
On the other hand, by Theorem~\ref{Thm:Inverse-Beckner1}, we have
\[
\langle T_{1/3} f' , h \rangle + \langle T_{1/3} (1-g), \bar{h}
\rangle \geq RHC(p_3,p_1) + RHC(1-p_3,1-p_2).
\]
Since $p_1 \leq 1-p_2$ and either $p_3$ or $1-p_3$ is not less
than $1/2$, we get
\begin{equation}
\langle T_{1/3} f' , h \rangle + \langle T_{1/3} (1-g), \bar{h}
\rangle \geq RHC(1/2,p_1) \geq 0.5 p_1^{1.12606},
\label{Eq:Modification2.3}
\end{equation}
where the last inequality follows from
Observation~\ref{Obs:Inv-Beckner} since $p_1 \leq 2^{-500000}$.
Combining Inequalities~(\ref{Eq:Modification2.2})
and~(\ref{Eq:Modification2.3}) we get
\[
\langle T_{1/3} f' , 1-g \rangle \leq p_1^{1.12706} =
(2p_1^{0.001}) (0.5 p_1^{1.12606}) \leq 0.5 (\langle T_{1/3} f' ,
h \rangle + \langle T_{1/3} (1-g), \bar{h} \rangle).
\]
Finally,
\begin{align}
\begin{split}
P(F) &= \langle T_{1/3} f' , h \rangle + \langle T_{1/3} (1-g),
\bar{h} \rangle - \langle T_{1/3} f' , 1-g \rangle \\
&\geq 0.5 (\langle T_{1/3} f' , h \rangle + \langle T_{1/3} (1-g), \bar{h}
\rangle) \geq 0.5 \Big(RHC(p_3,p_1) + RHC(1-p_3,1-p_2) \Big) \\
&\geq 0.5 \max \Big( RHC(1-p_2,1-p_2), RHC(1/2,p_1) \Big),
\end{split}
\end{align}
where the last inequality holds since $1-p_3 \geq 1-p_2$. The assertion of the
lemma follows now from Inequalities~(\ref{Eq3.4}) and~(\ref{Eq3.5}).

\subsubsection{Case 2b: $(1-p_2) > p_1^{0.45412}$ and
$\bar{p}_3 \geq p_1^{0.2002}$. }

The upper bound in this case is the same as in Case~2a:
\begin{equation}
\langle T_{1/3} f' , 1-g \rangle \leq p_1^{0.9} (1-p_2)^{0.5}.
\label{Eq:Modification2.4}
\end{equation}
For the lower bound, we use the reverse hypercontractive
inequality for the term $\langle T_{1/3} (1-g), \bar{h} \rangle$,
and get
\begin{equation}
\langle T_{1/3} (1-g), \bar{h} \rangle \geq RHC(1-p_3,1-p_2) \geq
(1-p_3)(1-p_2)^2 \geq p_1^{0.2002} (1-p_2)^2,
\label{Eq:Modification2.5}
\end{equation}
where the second inequality follows from
Observation~\ref{Obs:Inv-Beckner}, and the third inequality
follows from the assumption $\bar{p}_3 \geq p_1^{0.2002}$.
Combination of Inequality~(\ref{Eq:Modification2.4}) with
Inequality~(\ref{Eq:Modification2.5}) yields:
\begin{align*}
\langle T_{1/3} f' , 1-g \rangle &\leq p_1^{0.9} p_2^{0.5} = \Big(
p_1^{0.6998} p_2^{-1.5} \Big) \Big(p_1^{0.2002} (1-p_2)^2 \Big) \\
&\leq p_1^{0.01862} \langle T_{1/3} (1-g), \bar{h} \rangle \leq 0.5
\langle T_{1/3} (1-g), \bar{h} \rangle,
\end{align*}
where the second to last inequality follows from the assumption
$1-p_2 \geq p_1^{0.45412}$, and the last inequality follows since
$p_1 \leq 2^{-500000}$. Finally, if $p_3 \geq 1/2$ then $\langle
T_{1/3} f' , h \rangle \geq RHC(1/2,p_1)$, and otherwise,
$\langle T_{1/3} (1-g), \bar{h} \rangle \geq RHC(1/2,p_1)$. In both
cases,
\begin{align}
\begin{split}
P(F) &= \langle T_{1/3} f' , h \rangle + (\langle T_{1/3} (1-g),
\bar{h} \rangle - \langle T_{1/3} f' , 1-g \rangle) \\
&\geq \langle T_{1/3} f' , h \rangle + 0.5 \langle T_{1/3} (1-g),
\bar{h} \rangle \\
&\geq 0.5 \max \Big( RHC(1-p_2,1-p_2), RHC(1/2,p_1) \Big),
\end{split}
\end{align}
and the assertion follows.

\subsubsection{Case 2c: $p_3 \leq p_1^{0.2002}$.}

In this case we use another modification of Kalai's
formula~(\ref{Eq3.1}), resulting from the following
modification of Equation~(\ref{Eq:Modification2.1}):
\begin{align}\label{Eq:Modification2.6}
\begin{split}
P(F) &= \langle T_{1/3} f' , h \rangle + \langle T_{1/3} (1-g),
\bar{h} \rangle - \langle T_{1/3} f' , 1-g \rangle \\
&= \langle T_{1/3} f' , h \rangle + (\langle T_{1/3} (1-g), 1
\rangle - \langle T_{1/3} (1-g), 1-\bar{h} \rangle) - \langle
T_{1/3} f' , 1-g \rangle \\
&= (1-p_2) + \langle T_{1/3} f' , h \rangle - \langle T_{1/3}
(1-g), 1-\bar{h} \rangle - \langle T_{1/3} f' , 1-g \rangle.
\end{split}
\end{align}
By Proposition~\ref{Prop:Upper-Bound1},
\[
\langle T_{1/3} (1-g) , 1-\bar{h} \rangle \leq (1-p_2)^{0.9}
p_3^{0.5} \leq (1-p_2)^{0.9} p_1^{0.1001} \leq (1-p_2)^{1.0001}.
\]
Similarly,
\[
\langle T_{1/3} f' , 1-g \rangle \leq p_1^{0.9} (1-p_2)^{0.5} \leq
(1-p_2)^{1.4}.
\]
Hence,
\[
\langle T_{1/3} (1-g), 1-\bar{h} \rangle + \langle T_{1/3} f' ,
1-g \rangle \leq (1-p_2)^{1.0001} + (1-p_2)^{1.4} \leq
2(1-p_2)^{1.0001} \leq 0.5 (1-p_2),
\]
where the last inequality follows since
\[
(1-p_2)^{0.0001} \leq p_3^{0.0001} \leq p_1^{0.2002 \cdot 0.0001}
<1/2.
\]
Finally, by Equation~(\ref{Eq:Modification2.6}),
\begin{align}
\begin{split}
P(F) &\geq (1-p_2) - \langle T_{1/3} (1-g), 1-\bar{h} \rangle - \langle
T_{1/3} f' , 1-g \rangle \\
&\geq 0.5(1-p_2) \geq 0.5 \max \Big( RHC(1-p_2,1-p_2), RHC(1/2,p_1) \Big),
\end{split}
\end{align}
as asserted.

\subsubsection{Case 2d: $1-p_3 \leq p_1^{0.2002}$}

We use yet another modification of Equation~(\ref{Eq:Modification2.1}):
\begin{align}\label{Eq:Modification2.7}
\begin{split}
P(F) &= \langle T_{1/3} f' , h \rangle + \langle T_{1/3} (1-g),
\bar{h} \rangle - \langle T_{1/3} f' , 1-g \rangle \\
&= (\langle T_{1/3} f' , 1 \rangle - \langle T_{1/3} f' , 1-h
\rangle ) + \langle T_{1/3} (1-g), \bar{h} \rangle - \langle
T_{1/3} f' , 1-g \rangle \\
&= p_1 - \langle T_{1/3} f' , 1-h \rangle + \langle T_{1/3} (1-g),
\bar{h} \rangle - \langle T_{1/3} f' , 1-g \rangle.
\end{split}
\end{align}
Similarly to Case~2c, we have
\[
\langle T_{1/3} f' , 1-h \rangle \leq p_1^{0.9} (1-p_3)^{0.5} \leq
p_1^{1.0001},
\]
and $\langle T_{1/3} f' , 1-g \rangle \leq p_1^{1.4}$, and thus,
\[
\langle T_{1/3} f' , 1-h \rangle + \langle T_{1/3} f' , 1-g
\rangle \leq p_1^{1.0001} + p_1^{1.4} \leq 0.5 p_1.
\]
Therefore,
\begin{align}
\begin{split}
P(F) &= p_1 - \langle T_{1/3} f' , 1-h \rangle + \langle
T_{1/3} (1-g), \bar{h} \rangle - \langle T_{1/3} f' , 1-g \rangle \\
&\geq 0.5p_1 + \langle T_{1/3} (1-g), \bar{h} \rangle  \geq 0.5
\max \Big( RHC(1-p_2,1-p_2), RHC(1/2,p_1) \Big).
\end{split}
\end{align}
This completes the proof of Lemma~\ref{Lemma:Main}.
\end{proof}

\section{Proof of Theorems~\ref{Thm:Main1} and~\ref{Thm:Main2}}
\label{sec:Proof}

In this section we present the proofs of Theorems~\ref{Thm:Main1}
and~\ref{Thm:Main2}. The proofs are based on
Lemma~\ref{Lemma:Main}, but also rely heavily on several
components of Mossel's proof of his quantitative version of
Arrow's theorem cited in
Section~\ref{sec:sub:preliminaries-mossel}. The general structure
of both proofs is as follows:
\begin{enumerate}
\item We consider first GSWFs on three alternatives, and examine
several cases:
\begin{enumerate}
\item If $D_1(F)$ (resp., $D_2(F)$) is greater than a fixed
constant, we deduce the assertion from Theorem~\ref{Thm:Mossel1}.

\item If $F$ is close to a GSWF that always ranks one of the
candidates at the top/bottom (resp., if at least one of the
Boolean choice functions of $F$ is close to a constant function),
we deduce the assertion from Lemma~\ref{Lemma:Main}.

\item If $F$ (resp., one of the Boolean choice functions of $F$)
is close to a dictatorship of the $i$'th voter, we split $F$ into
six GSWFs $\{F^{\sigma}\}_{\sigma \in S_3}$ according to the
preferences of the $i$'th voter. We further subdivide this case
into two cases:
\begin{itemize}
\item If for all $\sigma \in S_3$, $D_1(F^{\sigma})$ (resp.,
$D'_2(F^{\sigma})$) is small, we get a contradiction (resp., show
directly that $P(F)$ cannot be small).

\item If there exists $\sigma_0 \in S_3$ such that
$D_1(F^{\sigma_0})$ (resp., $D'_2(F^{\sigma_0})$) is not small, we
deduce the assertion by applying Lemma~\ref{Lemma:Main} to the
GSWF $F^{\sigma_0}$.
\end{itemize}
\end{enumerate}

\item We leverage the result to GSWFs on $k$ alternatives, for all
$k \geq 3$. In the proof of Theorem~\ref{Thm:Main1} this requires
the reduction technique of Theorem~\ref{Thm:Leveraging}, and in
the proof of Theorem~\ref{Thm:Main2}, the generalization is
immediate.
\end{enumerate}

\noindent Since the proofs differ in many of the details, we
present them separately. Throughout this section, we use the
notations defined at the beginning of Section~\ref{sec:Lemma}.

\subsection{Proof of Theorem~\ref{Thm:Main1}}

\begin{theorem}\label{Thm4.1}
There exists an absolute constant $C$ such that for any GSWF $F$
on three alternatives that satisfies the IIA condition, if the
probability of non-transitive outcome in $F$ is at most
\[
\delta(\epsilon)= \min(C,\frac{1}{50000} \cdot \epsilon^3),
\]
then $D_1(F) \leq \epsilon$.
\end{theorem}

\begin{proof}
It is clearly sufficient to prove that for any $\epsilon>0$, if
$D_1(F)=\epsilon$, then $P(F) \geq \min(C,\frac{1}{50000} \cdot
\epsilon^3)$, for a universal constant $C$. We shall prove this
for
\begin{equation}\label{Eq:Constant}
C=\exp \left(-\frac{C'}{\left(2^{-500003}\right)^{21}} \right),
\end{equation}
where $C'$ is the constant in Mossel's Theorem~\ref{Thm:Mossel1}.

\medskip

\noindent Let $F$ be a GSWF on three alternatives satisfying the
IIA conditions, and denote the choice functions of $F$ by $f,g,$
and $h$, as in the proof of Lemma~\ref{Lemma:Main}. If $D_1(F)
\geq 2^{-500003}$, then by Theorem~\ref{Thm:Mossel1}, $P(F) \geq
C$. Thus, we may assume that $D_1(F) < 2^{-500003}$.

\medskip

\noindent Let $G \in \mathcal{F}_3(n)$ satisfy $\Pr[F \neq G] =
D_1(F)$ (such element exists by the definition of the distance
$D_1(F)$). Denote the Boolean choice functions of $G$ by $f',g'$,
and $h'$. By Theorem~\ref{Thm:Mossel3}, $G$ either always ranks
one alternative at the top/bottom or is a dictatorship. If $G$
always ranks one alternative at the top/bottom, then at least two
of the functions $f',g',$ and $h'$ are constant. Assume w.l.o.g.
that $f'$ and $g'$ are constant. Since
\[
D_1(F)=\Pr[F \neq G] \geq \max \left(\Pr[f \neq f'], \Pr[g \neq
g'], \Pr[h \neq h'] \right),
\]
it follows that either $\mathbb{E}[f] \leq D_1(F)$ or
$\mathbb{E}[f] \geq 1-D_1(F)$, and similarly for $g$. This implies
that $D'_2(F) \leq D_1(F)<2^{-500003}$, and thus we can apply
Lemma~\ref{Lemma:Main} to $F$ and get
\[
P(F) \geq \frac{1}{10} \cdot RHC(D_1(F)/2,D_1(F)/2) \geq
\frac{1}{10} \cdot (D_1(F)/2)^3 > \frac{1}{50000} \cdot D_1(F)^3,
\]
as asserted. Thus, we may assume that $G$ is a dictatorship.

\medskip

\noindent The following part of the proof is similar to the proof
of Theorem~7.1 in~\cite{Mossel-Arrow}. W.l.o.g., we assume that
the output of $G$ is determined by the first voter. We ``split''
the choice functions according to the first voter. Let
\[
f^0(x_2,x_3,\ldots,x_n)=f(0,x_2,x_3,\ldots,x_n), \qquad
f^1(x_2,x_3,\ldots,x_n)=f(1,x_2,x_3,\ldots,x_n),
\]
and similarly for $g$ and $h$. Furthermore, for any profile
$(\sigma_1,\sigma_2,\ldots,\sigma_n) \in S_3^n$, denote
\[
F^{\sigma_1}(\sigma_2,\sigma_3,\ldots,\sigma_n)=
F(\sigma_1,\sigma_2,\sigma_3,\ldots,\sigma_n),
\]
and similarly for $G$. The Boolean choice functions of
$F^{\sigma}$ are $f^{a_1},g^{a_2},$ and $h^{a_3}$, where
$(a_1,a_2,a_3) \in \{0,1\}^3$ represents the preference $\sigma$
of the first voter (note that only six of the eight possible
combinations of $(a_1,a_2,a_3)$ represent elements of $S_3$).
Denote by $\bar{f}^{a_1},\bar{g}^{a_2},\bar{h}^{a_3}$ the choice
functions of $G^{\sigma}$. Since $G$ is a dictatorship of the
first voter, the functions $\bar{f}^{a_1},\bar{g}^{a_2},$ and
$\bar{h}^{a_3}$ are constant. Clearly, we have
\begin{equation}\label{Eq4.0}
D_1(F)=\Pr[F \neq G] = \frac{1}{6} \sum_{\sigma \in S_3}
\Pr[F^{\sigma} \neq G^{\sigma}],
\end{equation}
and thus, for all $\sigma \in S_3$,
\[
\Pr[F^{\sigma} \neq G^{\sigma}] \leq 6D_1(F).
\]
Since
\[
\Pr[F^{\sigma} \neq G^{\sigma}] \geq \max \left(\Pr[f^{a_1} \neq
\bar{f}^{a_1}], \Pr[g^{a_2} \neq \bar{g}^{a_2}], \Pr[h^{a_3} \neq
\bar{h}^{a_3}] \right),
\]
and since $G^{\sigma}$ is constant, this implies that
\begin{equation}\label{Eq4.1}
\mathbb{E}[f^{a_1}] \leq 6D_1(F) \qquad \mbox{or} \qquad
\mathbb{E}[f^{a_1}] \geq 1-6D_1(F),
\end{equation}
and similarly for $g^{a_2}$ and $h^{a_3}$.

\medskip

\noindent The rest of the proof is divided into two cases:
\begin{itemize}
\item \textbf{Case A:} For all $\sigma \in S_3$ we have
$D_1(F^{\sigma}) \leq D_1(F)/4$.

\item \textbf{Case B:} There exists $\sigma_0 \in S_3$ such that
$D_1(F^{\sigma_0}) > D_1(F)/4$.
\end{itemize}

\noindent We first show that Case~A leads to a contradiction by
constructing a GSWF $G' \in \mathcal{F}_3(n)$ such that $\Pr[F
\neq G'] < D_1(F)$. Then we show that in Case~B, the assertion of
the theorem follows by applying Lemma~\ref{Lemma:Main} to the
function $F^{\sigma_0}$.

\medskip

\noindent \textbf{Case A:} Consider a GSWF $G'$ whose choice
functions $f'', g'',$ and $h''$ are defined as follows: For $a_1
\in \{0,1\}$,
\[
f''(a_1,x_2,\ldots,x_n)= \left\lbrace
  \begin{array}{c l}
    1 \mbox{ (constant) }, & \mathbb{E}[f^{a_1}] \geq 1-D_1(F)/4, \\
    0 \mbox{ (constant) }, & \mathbb{E}[f^{a_1}] \leq D_1(F)/4, \\
    f^{a_1}, & \mbox{otherwise},
  \end{array}
\right.
\]
and similarly for $g'$ and $h'$. We claim that the output of $G'$
is always transitive, and thus $G' \in \mathcal{F}_3 (n)$. Indeed,
by assumption, for any $\sigma \in S_3$, there exists
$\bar{G}^{\sigma} \in \mathcal{F}_3 (n-1)$ such that
$\Pr[F^{\sigma} \neq \bar{G}^{\sigma}] \leq D_1(F)/4$. The GSWF
$\bar{G}^{\sigma}$ cannot be a dictatorship since by
Equation~(\ref{Eq4.1}), the choice functions $f^{a_1},g^{a_2},$
and $h^{a_3}$ of $F^{\sigma}$ satisfy
\[
\mathbb{E}[f^{a_1}] \leq 6D_1(F) \qquad \mbox{or} \qquad
\mathbb{E}[f^{a_1}] \geq 1-6D_1(F),
\]
and thus, for any dictatorship $H$,
\[
\Pr[F^{\sigma} \neq H] \geq 1/2-6D_1(F)> 1/2-2^{500000}.
\]
Therefore, $\bar{G}^{\sigma}$ always ranks one alternative at the
top/bottom. Denote the choice functions of $\bar{G}^{\sigma}$ by
$\tilde{f},\tilde{g},$ and $\tilde{h}$, and assume w.l.o.g. that
$\bar{G}^{\sigma}$ always ranks alternative $1$ at the top, and
thus $\tilde{f}=1$ and $\tilde{h}=0$. Since
\[
D_1(F)/4 \geq \Pr[F^{\sigma} \neq \bar{G}^{\sigma}] \geq \max
\left( \Pr[f^{a_1} \neq \tilde{f}],\Pr[g^{a_2} \neq
\tilde{g}],\Pr[h^{a_3} \neq \tilde{h}] \right),
\]
it follows that
\[
\mathbb{E}[f^{a_1}] \geq 1-D_1(F)/4, \qquad \mbox{ and } \qquad
\mathbb{E}[h^{a_3}] \leq D_1(F)/4.
\]
Hence, by the definition of $G'$, its choice functions satisfy
$f''=1$ and $h''=0$, which means that $G'$ always ranks
alternative $1$ at the top, and is thus always transitive.

\medskip

\noindent Therefore, $G' \in \mathcal{F}_3(n)$, and on the other
hand, we have
\[
\Pr[F \neq G'] \leq \Pr[f \neq f''] + \Pr[g \neq g''] + \Pr[h \neq
h''] \leq 3 \cdot D_1(F)/4 < D_1(F),
\]
contradicting the definition of $D_1(F)$.

\bigskip

\noindent \textbf{Case B:} Let $\sigma_0 \in S_3$ be such that
$D_1(F^{\sigma_0}) > D_1(F)/4$. By Equation~(\ref{Eq4.1}), the
choice functions $f^{a_1},g^{a_2},h^{a_3}$ of $F^{\sigma_0}$
satisfy
\[
\mathbb{E}[f^{a_1}] \leq 6D_1(F) \qquad \mbox{or} \qquad
\mathbb{E}[f^{a_1}] \geq 1-6D_1(F),
\]
and thus (in the notation of Lemma~\ref{Lemma:Main}),
$D'_2(F^{\sigma_0}) \leq 6D_1(F)< 2^{-500000}$. Hence, we can
apply Lemma~\ref{Lemma:Main} to the GSWF $G^{\sigma_0}$, and get
\[
P(F^{\sigma_0}) \geq \frac{1}{10} \cdot
RHC(D_1(F^{\sigma_0})/2,D_1(F^{\sigma_0})/2) \geq \frac{1}{10}
(D_1(F)/8)^3 = \frac{1}{5120} \cdot D_1(F)^3.
\]
Finally,
\[
P(F) = \frac{1}{6} \sum_{\sigma \in S_3} P(F^{\sigma}) \geq
\frac{1}{6} \cdot P(F^{\sigma_0}) > \frac{1}{50000} D_1(F)^3.
\]
This completes the proof of the theorem.
\end{proof}

\medskip

\noindent Theorem~\ref{Thm:Main1} follows immediately from
Theorem~\ref{Thm4.1} using Theorem~\ref{Thm:Leveraging} (the
generic reduction lemma of Mossel).

\subsection{Proof of Theorem~\ref{Thm:Main2}}

\begin{theorem}\label{Thm4.2}
There exists an absolute constant $C$ such that for any GSWF $F$
on three alternatives that satisfies the IIA condition, if the
probability of non-transitive outcome in $F$ is at most
\begin{equation}\label{Eq1.1}
\delta(\epsilon)= \min \left(C, \frac{1}{10000} \cdot
\epsilon^{\frac{9(\sqrt{\log_2
(1/\epsilon)}+1/3)^2}{8\log_2(1/\epsilon)}} \right),
\end{equation}
then $D_2(F) \leq \epsilon$.
\end{theorem}

\begin{proof}
By Equation~(\ref{Eq:RHC'}), it is sufficient to prove that for
any $\epsilon>0$, if $D_2(F)=\epsilon$, then
\[
P(F) \geq \min \left(C,\frac{1}{5000} \cdot RHC(1/2,\epsilon)
\right),
\]
for a universal constant $C$. We shall prove this for
\begin{equation}\label{Eq:Constant2}
C=\exp \left(-\frac{C'}{\left(2^{-500003}\right)^{21}} \right),
\end{equation}
where $C'$ is the constant in Mossel's Theorem~\ref{Thm:Mossel1}.
Let $F$ be a GSWF on three alternatives satisfying the IIA
conditions, and denote the choice functions of $F$ by $f,g,$ and
$h$, as in the proof of Lemma~\ref{Lemma:Main}.

\medskip

\noindent First we consider the case $D_2(F) \geq 2^{-500003}$. We
show that in general, $D_1(F) \geq D_2(F)$, and thus in this case
we have $D_1(F) \geq D_2(F) \geq 2^{-500003}$, which by
Theorem~\ref{Thm:Mossel1} implies that $P(F) \geq C$. Let $G \in
\mathcal{F}_3(n)$ satisfy $\Pr[F \neq G] = D_1(F)$, and denote the
Boolean choice functions of $G$ by $f',g'$, and $h'$. Clearly,
\begin{equation}\label{Eq4.3}
D_1(F)=\Pr[F \neq G] \geq \max \left(\Pr[f \neq f'], \Pr[g \neq
g'], \Pr[h \neq h'] \right).
\end{equation}
By Theorem~\ref{Thm:Mossel3}, $G$ either always ranks one
alternative at the top/bottom or is a dictatorship. In the first
case, at least two of the functions $f',g',$ and $h'$ are
constant, and thus Equation~(\ref{Eq4.3}) implies that at least
two of the functions $f,g,$ and $h$ are at most $D_1(F)$-far from
a constant function. In the latter case, the functions $f',g',$
and $h'$ are dictatorships, and thus Equation~(\ref{Eq4.3})
implies that $f,g,$ and $h$ are at most $D_1(F)$-far from a
dictatorship. Hence, in both cases,
\[
D_2(F)=\min_{1 \leq i<j \leq 3} \min_{G \in \mathcal G_2(n)}
\Pr[F_{ij} \neq G] \leq D_1(F),
\]
as asserted.

\medskip

\noindent Now we consider the case $D_2(F)<2^{-500003}$. Assume
w.l.o.g. that the minimal distance $\min_{G \in \mathcal G_2(n)}
\Pr[F_{ij} \neq G]$ is obtained by the choice function $f$, and
let $\tilde{f} \in \mathcal G_2(n)$ satisfy $\Pr[f \neq \tilde{f}]
= D_2(F)$. If $\tilde{f}$ is a constant function, then in the
notations of Lemma~\ref{Lemma:Main}, this implies that $D'_2(F) =
D_2(F)<2^{-500003}$, and thus we can apply Lemma~\ref{Lemma:Main}
to $F$ and get
\[
P(F) \geq \frac{1}{10} \cdot RHC(1/2,D'_2(F)) > \frac{1}{5000}
\cdot RHC(1/2,D_2(F)),
\]
as asserted. Thus, we may assume that $\tilde{f}$ is a
dictatorship.

\medskip

\noindent Assume w.l.o.g. that $\tilde{f}$ is a dictatorship of
the first voter. Define the functions $F^{\sigma},
f^0,f^1,g^0,g^1,h^0,$ and $h^1$ as in the proof of
Theorem~\ref{Thm4.1}, and let
\[
\tilde{f}^0(x_2,x_3,\ldots,x_n) = \tilde{f}(0,x_2,x_3,\ldots,x_n),
\qquad \mbox{and} \qquad \tilde{f}^1(x_2,x_3,\ldots,x_n) =
\tilde{f}(1,x_2,x_3,\ldots,x_n).
\]
Clearly, we have
\begin{equation}\label{Eq4.4}
D_2(F)=\Pr[f \neq \tilde{f}] = \frac{1}{2} (\Pr[f^0 \neq
\tilde{f}^0] + \Pr[f^1 \neq \tilde{f}^1]),
\end{equation}
and thus, for $a_1 \in \{0,1\}$,
\[
\Pr[f^{a_1} \neq \tilde{f}^{a_1}] \leq 2D_2(F).
\]
Since $\tilde{f}^{0}$ and $\tilde{f}^1$ are constant functions,
this implies that
\begin{equation}\label{Eq4.5}
\mathbb{E}[f^{a_1}] \leq 2D_2(F) \qquad \mbox{or} \qquad
\mathbb{E}[f^{a_1}] \geq 1-2D_2(F).
\end{equation}

\medskip

\noindent The rest of the proof is divided into two cases:
\begin{itemize}
\item \textbf{Case A:} For all $\sigma \in S_3$ we have
$D'_2(F^{\sigma}) \leq D_2(F)/4$.

\item \textbf{Case B:} There exists $\sigma_0 \in S_3$ such that
$D'_2(F^{\sigma_0}) > D_2(F)/4$.
\end{itemize}

\noindent \textbf{Case A:} In this case, for any $\sigma \in S_3$,
at least one of the choice functions of $F^{\sigma}$ is at most
$D_2(F)/4$-far from a constant function. Note that if $f^0$ is at
most $D_2(F)/4$-far from a constant function, then $f^1$ must be
at least $7D_2(f)/4$-far from a constant function, since
otherwise, $f$ is less than $D_2(F)$-far either from a constant
function or from a dictatorship, contradicting the definition of
$D_2(F)$. The same holds also for the pairs $(g^0,g^1)$ and
$(h^0,h^1)$. Thus, the only two possibilities are that either the
functions $f^1,g^1,h^1$ or the functions $f^0,g^0,h^0$ are
simultaneously at most $D_2(F)/4$-far from a constant function.
(For example, if $f^1,g^1,$ and $h^0$ are at most $D_2(F)/4$-far
from a constant function, then $f^0,g^0,$ and $h^1$ are at least
$7D_2(F)/4$-far from a constant function, and thus, for the
preference $\sigma=(0,0,1)$, we have $D'_2(F^{\sigma}) \geq
7D_2(F)/4$, a contradiction. The other possibilities are discarded
in a similar way). Assume w.l.o.g. that $f^1,g^1,$ and $h^1$ are
at most $D_2(F)/4$-far from a constant function. Furthermore,
since amongst the expectations
$\mathbb{E}[f^1],\mathbb{E}[g^1],\mathbb{E}[h^1]$, at least two
are close to one or at least two are close to zero, we can assume
w.l.o.g. that
\[
\Pr[f^1 \neq 1] \leq D_2(F)/4, \qquad \mbox{and} \qquad \Pr[g^1
\neq 1] \leq D_2(F)/4.
\]
Consider the GSWF $F^{\sigma_0}$ for the preference
$\sigma_0=(1,1,0)$. Since $h^0$ is at least $7D_2(F)/4$-far from
the constant zero function, it follows that
\begin{align*}
P(F^{\sigma}) &\geq \Pr_{\mbox{profile } \in (S_3)^{n-1}}
[(f^1,g^1,h^0)(profile)=(1,1,1)] \\
&\geq 7D_2(F)/4 - D_2(F)/4- D_2(F)/4 = 5D_2(F)/4,
\end{align*}
and thus,
\[
P(F) = \frac{1}{6} \sum_{\sigma \in S_3} P(F^{\sigma}) \geq
\frac{1}{6} \cdot P(F^{\sigma_0}) \geq \frac{5}{24} \cdot D_2(F)>
\frac{1}{5000} \cdot RHC(1/2,D_2(F)),
\]
as asserted.

\bigskip

\noindent \textbf{Case B:} Let $\sigma_0 \in S_3$ be such that
$D'_2(F^{\sigma_0}) > D_2(F)/4$. By Equation~(\ref{Eq4.5}), the
choice function $f^{a_1}$ of $F^{\sigma_0}$ satisfies
\[
\mathbb{E}[f^{a_1}] \leq 2D_2(F) \qquad \mbox{or} \qquad
\mathbb{E}[f^{a_1}] \geq 1-2D_2(F),
\]
and thus, $D'_2(F^{\sigma_0}) \leq 2D_2(F)< 2^{-500000}$. Hence,
we can apply Lemma~\ref{Lemma:Main} to the GSWF $G^{\sigma_0}$,
and get
\[
P(F^{\sigma_0}) \geq \frac{1}{10} \cdot
RHC(1/2,D'_2(F^{\sigma_0})) \geq \frac{1}{10} \cdot
RHC(1/2,D_2(F)/4) \geq \frac{1}{640} \cdot RHC(1/2,D_2(F)).
\]
Finally,
\[
P(F) = \frac{1}{6} \sum_{\sigma \in S_3} P(F^{\sigma}) \geq
\frac{1}{6} \cdot P(F^{\sigma_0}) > \frac{1}{5000} \cdot
RHC(1/2,D_2(F)).
\]
This completes the proof of the theorem.
\end{proof}

\medskip

\noindent The generalization to $k$ alternatives for all $k \geq
3$ follows immediately by applying Theorem~\ref{Thm4.2} to any
subset of three alternatives.

\section{Tightness of Results}
\label{sec:Tightness}

In this section we show that for GSWFs on three alternatives, the
assertions of Theorems~\ref{Thm:Main1} and~\ref{Thm:Main2} are
tight up to logarithmic factors. In all our examples below, the
Boolean choice functions $f,g,h$ of the GSWF $F$ are monotone
threshold functions, that is, functions of the form:
\begin{equation}\label{Eq5.0}
(f(x)=1) \Leftrightarrow \left( \sum_{i=1}^n x_i \geq l \right),
\end{equation}
for different values of $l$. We note that in~(\cite{Maj-Stablest},
Theorem~2.9), Mossel et al. showed that amongst {\it neutral}
GSWFs on three alternatives, a GSWF based on the majority rule is
the ``most rational'' in the asymptotic sense (i.e., has the least
probability of non-transitive outcome as the number of voters
tends to infinity). To some extent, our examples generalize this
result to general GSWFs on three alternatives. The examples show
that GSWFs based on monotone threshold Boolean choice functions
are ``close to be the most rational'' amongst GSWFs whose choice
functions have the same expectations, in the sense that their
probability of non-transitive outcome is logarithmic close to the
lower bound. In fact, we conjecture that such GSWFs are indeed the
{\it most} rational amongst GSWFs whose choice functions have the
same expectations. However, such exact result is not known even
for neutral GSWFs.

\medskip

\noindent We use the following proposition of Mossel et
al.~\cite{Mossel-Inverse}, showing that Borell's reverse
Bonami-Beckner inequality is essentially tight for diametrically
opposed Hamming balls. Since we use the proposition only for noise
of rate $\epsilon=1/3$, we state it in this particular case.
\begin{theorem}[~\cite{Mossel-Inverse}, Proposition~3.9]
\label{Thm5.1}
Fix $s,t>0$, and let $f_n,g_n:\{0,1\}^n \rightarrow
\{0,1\}$ be defined by
\[
(f_n(x)=1) \Leftrightarrow \left( \sum_{i=1}^n x_i \leq
\frac{n}{2}-\frac{s}{2}\sqrt{n} \right), \qquad \mbox{and} \qquad
(g_n(x)=1) \Leftrightarrow \left( \sum_{i=1}^n x_i \geq
\frac{n}{2}+\frac{t}{2}\sqrt{n} \right).
\]
Then
\begin{equation}\label{Eq5.1}
\lim_{n \rightarrow \infty} \sum_{S \subset \{1,\ldots,n\}} \left(
\frac{1}{3} \right)^{|S|} \hat f_n(S) \hat g_n(S) \leq
\frac{\sqrt{8/9}}{2 \pi s(s/3+t)} \exp \left(-\frac{1}{2}
\frac{s^2+2st/3+t^2}{8/9} \right).
\end{equation}
\end{theorem}

\medskip

\noindent In order to show the tightness of
Theorem~\ref{Thm:Main1}, we fix a constant $\epsilon>0$ and define
the choice functions according to Equation~(\ref{Eq5.0}), choosing
the values of $l$ such that
\[
\mathbb{E}[f]=0, \qquad \mathbb{E}[g]=1-\epsilon, \qquad
\mathbb{E}[h]=1-\epsilon.
\]
It is clear that $D_1(F)=\epsilon$. By
Equation~(\ref{Eq:Modification2.1}),
\[
P(F)=\langle T_{1/3} (1-g), \bar{h} \rangle.
\]
By our construction, the pair of functions $(1-g,\bar{h})$ is of
the form considered in Theorem~\ref{Thm5.1}, with $s=t \approx
\sqrt{2\log(1/\epsilon)}$, and thus by the theorem, for $n$
sufficiently large,
\[
P(F)=\langle T_{1/3} (1-g), \bar{h} \rangle \leq
\frac{\sqrt{8/9}}{2 \pi s(s/3+t)} \exp \left(-\frac{1}{2}
\frac{s^2+2st/3+t^2}{8/9} \right) \approx C \epsilon^3
\log(1/\epsilon).
\]
The lower bound asserted by Theorem~\ref{Thm:Main1} is $P(F) \geq
C' \cdot \epsilon^3$, and thus the example shows the tightness of
the assertion up to logarithmic factors.

\medskip

\noindent The tightness of Theorem~\ref{Thm:Main2} is shown
similarly, with choice functions chosen such that
\[
\mathbb{E}[f]=\epsilon, \qquad \mathbb{E}[g]=1-\epsilon, \qquad
\mathbb{E}[h]=1/2.
\]
It is clear that $D_2(F)=\epsilon$, and by
Equation~(\ref{Eq:Modification2.1}),
\[
P(F) \leq \langle T_{1/3} f', h \rangle + \langle T_{1/3} (1-g),
\bar{h} \rangle.
\]
The pairs $(f',h)$ and $(1-g,\bar{h})$ are both of the form
considered in Theorem~\ref{Thm5.1}, and application of the theorem
to both of them yields tightness up to a logarithmic factor, like
in the previous case.

\medskip

\noindent Finally, we note that while the examples above deal with
GSWFs whose choice functions have constant expectation, it also
makes sense to consider choice functions whose expectation tends
to zero, as $n$ (the number of voters) tends to infinity. In
particular, one may ask what is the {\it least possible}
probability of non-transitive outcome, as function of $n$, for
GSWFs with $D_1(F)>0$ or $D_2(F)>0$.  It appears that the question
is of interest mainly for $D_2(F)$, as for $D_1(F)$, one can
easily check that the minimal possible probability of $6^{-n}$ is
obtained by a GSWF whose choice functions are chosen according to
Equation~(\ref{Eq5.0}), such that
\[
\mathbb{E}[f]=0, \qquad \mathbb{E}[g]=1-2^{-n}, \qquad
\mathbb{E}[h]=1-2^{-n}.
\]
For $D_2(F)$, it was shown in~\cite{Keller-Choice} that for a GSWF
whose choice functions are chosen according to
Equation~(\ref{Eq5.0}), such that
\[
\mathbb{E}[f]=2^{-n}, \qquad \mathbb{E}[g]=1-2^{-n}, \qquad
\mathbb{E}[h]=1/2,
\]
we have $P(F) \leq 0.471^n$. Furthermore, it was conjectured that
this is the most rational GSWF on three alternatives that
satisfies the assumptions of Arrow's theorem (and in particular,
the minimal possible probability $6^{-n}$ is not obtained). Our
results show that this function is at least ``close'' to be the
most rational, as by Theorem~\ref{Thm:Main2}, for any GSWF $F$
such that $D_2(F)>0$, we have
\[
P(F) \geq C \cdot RHC(1/2,2^{-n}) \approx C \cdot 0.458^n.
\]

\section{Questions for Further Research}
\label{sec:Open-Questions}

We conclude this paper with several open problems related to our
results.

\begin{itemize}

\item Our main lemma (Lemma~\ref{Lemma:Main}) gives an essentially
tight lower bound on the probability of non-transitive outcome for
GSWFs in which at least one of the Boolean choice functions is
``close'' to a constant function. In the case where the distance
from constant functions is greater than a fixed constant, our
technique fails, and we use Mossel's theorem~\cite{Mossel-Arrow}
instead. As a result, the constant multiplicative factor in the
assertions of Theorems~\ref{Thm:Main1} and~\ref{Thm:Main2} is
extremely small, and clearly non-optimal. It will be interesting
to find a ``direct'' proof also for GSWFs whose Boolean choice
functions are ``far'' from constant functions, thus removing the
reliance of the proof on the non-linear invariance principle (used
in Mossel's argument) that seems ``unnatural'' in our context, and
improving the constant factor.

\item While the results of Kalai~\cite{Kalai-Choice} and
Mossel~\cite{Mossel-Arrow} hold also for more general
distributions of the individual preferences called ``even product
distributions'' or ``symmetric distributions''
(see~\cite{Keller-Choice,Mossel-Arrow}), our proof does not extend
directly to such distributions. The reason is that for highly
biased distributions of the preferences, the lower bound obtained
by Borell's reverse Bonami-Beckner inequality is weaker, and
cannot ``beat'' the upper bound obtained by the Bonami-Beckner
inequality. Thus, obtaining a tight quantitative version of
Arrow's theorem for general even product distributions of the
preferences is an interesting open problem.

\item We believe that GSWFs whose Boolean choice functions are
monotone threshold functions are the most rational amongst GSWFs
whose choice functions have the same expectations, not only in the
asymptotic sense, but also for any particular (large enough) $n$.
However, this conjecture seems quite challenging, as it includes
the Majority is Stablest conjecture (whose proof by Mossel et
al.~\cite{Maj-Stablest} holds only in the limit as $n \rightarrow
\infty$).

\item Another direction of research is using our techniques to
obtain quantitative versions of other theorems in social choice
theory. In~\cite{FKN2}, Friedgut et al. presented a quantitative
version of the Gibbard-Satterthwaite
theorem~\cite{Gibbard,Satterthwaite} for neutral GSWFs on three
alternatives. Recently, Isaksson et al.~\cite{IKM} generalized the
result of~\cite{FKN2} to neutral GSWFs on $k$ alternatives, for
all $k \geq 4$. One of the main ingredients in the proof
of~\cite{FKN2} is Kalai's quantitative Arrow theorem for neutral
GSWFs. It seems interesting to find out whether our quantitative
version of Arrow's theorem can lead to a quantitative
Gibbard-Satterthwaite theorem for general GSWFs (without the
neutrality assumption).

\item Finally, our results (as well as the previous results of
Kalai~\cite{Kalai-Choice} and Mossel~\cite{Mossel-Arrow}) apply
only to GSWFs that satisfy the IIA condition, since such GSWFs can
be represented by their Boolean choice functions, which allows to
use the tools of discrete harmonic analysis. It will be very
interesting to find a quantitative version of Arrow's theorem that
will not assume the IIA condition, but rather will relate the
probability of non-transitive outcome to the distance of the GSWF
from satisfying IIA.

\end{itemize}

\section{Acknowledgements}

It is a pleasure to thank Gil Kalai and Elchanan Mossel for
numerous fruitful discussions that motivated our work.

\end{document}